\documentclass[11pt,reqno]{amsart}
\usepackage{fullpage}
\usepackage{amsmath}
\usepackage{amsfonts}
\usepackage{amsthm}
\usepackage{amssymb}
\usepackage{color}
\usepackage{enumerate}
\usepackage{graphicx}
\usepackage{subfigure}
\usepackage{mathrsfs}
\usepackage{mathtools}
\usepackage{multicol}
\usepackage{tikz}
\usepackage{eufrak}

\usepackage{hyperref}

\newtheorem{theorem}{Theorem}

\newtheorem{lemma}[theorem]{Lemma}
\newtheorem{corollary}[theorem]{Corollary}

\theoremstyle{definition}
\newtheorem{remark}[theorem]{Remark}

\numberwithin{equation}{section}
\numberwithin{figure}{section}
\numberwithin{theorem}{section}

\newcommand{\R}{\mathbb{R}}
\newcommand{\E}{\mathbb{E}}

\newcommand{\cL}{\mathcal{L}}
\newcommand{\cG}{\mathcal{G}}

\newcommand{\eps}{\varepsilon}

\newcommand{\om}{\omega}
\newcommand{\Om}{\Omega}
\newcommand{\cF}{\mathcal{F}}
\newcommand{\bP}{\mathbb{P}}
\newcommand{\gam}{\gamma}
\newcommand{\ol}{\overline}

\newcommand{\frakg}{\mathfrak{g}}
\begin{document}

\title{Fluctuations in the Homogenization of Semilinear Equations with Random Potentials}




\author[G. Bal]{Guillaume Bal}
\address{Department of Applied Physics and Applied Mathematics\\
Columbia University\\ 500 W. 120th Street, NY 10025, USA}
\email{gb2030@columbia.edu}

\author[W. Jing]{Wenjia Jing}
\address{Department of Mathematics\\
The University of Chicago\\ 5734 S. University Avenue Chicago, IL 60637, USA}
\email{wjing@math.uchicago.edu}


\maketitle

\begin{abstract}

We study the stochastic homogenization and obtain a random fluctuation theory for semilinear elliptic equations with a rapidly varying random potential. To first order, the effective potential is the average potential and the nonlinearity is not affected by the randomness. We then study the limiting distribution of the properly scaled homogenization error (random fluctuations) in the space of square integrable functions, and prove that the limit is a Gaussian distribution characterized by the homogenized solution, the Green's function of the linearized equation around the homogenized solution, and by the integral of the correlation function of the random potential. These results enlarge the scope of the framework that we have developed for linear equations to the class of  semilinear equations.

\medskip

\medskip

\smallskip

\end{abstract}



\section{Introduction}

\subsection{Motivation and overview}
We study the asymptotic stochasticity of solutions to the semilinear elliptic equation with random potential
\begin{equation} \label{e.rpde}
\left\{
\begin{aligned}
&-\Delta u^\eps + q_\eps(x,\omega) u^\eps + f(u^\eps) = g(x), \quad&\quad& x \in D,\\
&u^\eps = 0, \quad&\quad& x \in \partial D,
\end{aligned}
\right.
\end{equation}
and characterize the limiting distribution of the random fluctuations (correction to homogenization). Here, $D$ is an open bounded domain in $\R^d$, $d = 2,3$, with smooth (say $C^2$) boundary $\partial D$. The (linear) potential $q_\eps(x,\omega)$ is highly oscillatory and is of the form $q(\frac{x}{\eps},\omega)$, where $0 < \eps \ll 1$ denotes the scale of oscillations. We assume that $q(x,\omega)$, the potential function before scaling, is a stationary ergodic random field on a probability space $(\Om,\cF,\bP)$. The nonlinearity $f$ is assumed to satisfy conditions that, essentially, guarantee that for each $\om \in \Om$ and $g \in H^{-1}(D)$, the above equation admits a unique weak solution.

Under mild conditions, such as, e.g., stationarity and ergodicity, on the random potential $q(x,\omega)$ and under natural structural conditions on the nonlinearity $f(u)$, we prove in Theorem \ref{t.homog} below that the above semilinear problem homogenizes as $\eps \to 0$ and that the effective potential is in fact given by averaging and the nonlinearity term is not changed. In other words, $u^\eps$ converges strongly in $L^2(D)$ and weakly in $H^1(D)$, for a.e. $\om \in \Omega$, to the solution $u$ of the effective equation
\begin{equation} \label{e.hpde}
\left\{
\begin{aligned}
&-\Delta u + \ol{q} u + f(u) = g(x), \quad&\quad& x \in D,\\
&u = 0, \quad&\quad& x \in \partial D,
\end{aligned}
\right.
\end{equation}
where $\ol{q} = \E q$ is the average of $q(x,\om)$. The homogenized equation hence is deterministic and has smooth (non-oscillatory) coefficients.

The main contributions of this paper is an estimate of the size of the homogenization error $u^\eps - u$, say in the $L^2(\Om,L^2(D))$ norm, and more importantly, a study of the law (probability distribution) of the random fluctuations $u^\eps - u$, viewing them as a random element in the Hilbert space $L^2(D)$. Such asymptotic estimates of stochasticity beyond the homogenization limit find applications in uncertainty quantification and Bayesian formulation of PDE-based inverse problems; see e.g. \cite{BR08,NP09}.

The study of the limiting distribution of the homogenization error goes back to \cite{FOP-82}, where the Laplace operator with a random potential formed by Poisson bumps was considered. General random potential with short range correlations was considered recently in \cite{B-CLH-08}, and in \cite{BJ-DCDS-10, BJ-CMS-11} for other non-oscillatory differential operators with random potential. Long range correlated random potential was considered in \cite{BGGJ12_AA}. When random elliptic differential operators are considered, the limiting distribution of the homogenization error was obtained in \cite{BP-99} for short-range correlated elliptic coefficients, and in \cite{BGMP-AA-08} in the long-range correlated case; all in the one dimensional setting; we refer to \cite{GNO15,GO14,MO14,GM15,MN15} for important recent advances in this direction in higher spatial dimensions under strong structural assumptions on the probability distribution of the coefficients. 

The main results of this paper show that the general framework developed by the authors in \cite{B-CLH-08,BJ-CMS-11,BGGJ12_AA,Jing15}, for linear equations with random potential, essentially applies to the semilinear equation \eqref{e.rpde}. Let us briefly explain the main ideas. In the linear framework, e.g., for \eqref{e.rpde} with $f = 0$, let us denote the fundamental solution operator (Green's operator) of the Dirichlet problem for $-\Delta + \ol{q}$ by $\cG$. Then the homogenization error $u^\eps - u$ admits the following expansion formula:
\begin{equation}
u^\eps - u = -\cG \nu_\eps u + \cG \nu_\eps \cG \nu_\eps u + \cG \nu_\eps \cG \nu_\eps (u^\eps - u).
\label{e.lexpan}
\end{equation}   
Here, $\nu_\eps = q_\eps - \ol{q}$ denotes the random fluctuation of the potential. We also have $u^\eps - u = -\cG^\eps \nu_\eps u$, where $\cG^\eps$ is the solution operator associated to $-\Delta + q_\eps$. As long as $q_\eps > -\lambda_1$, the first eigenvalue of the Laplace operator, we obtain  a bound on $u^\eps - u$ in $L^2(\Om,L^2(D))$ provided that we can estimate the correlation function of $\nu_\eps$. The leading term in the right-hand side of \eqref{e.lexpan} is the first one. The last term is smaller because both $u^\eps - u$ and the operator $\cG \nu_\eps \cG$ are small. The middle term is also of smaller order but for a different reason: one can verify that $\cG \nu_\eps \cG \nu_\eps u - \E(\cG \nu_\eps \cG \nu_\eps u)$ is small, provided we can estimate fourth order moments of $\nu_\eps$, and $\E (\cG \nu_\eps \cG \nu_\eps u)$ is also small in dimensions $d = 2,3$.

For the nonlinear equation, \eqref{e.lexpan} has to be modified. In fact, as was already pointed out in \cite{FOP-82} without technical details, the leading term of $u^\eps - u$ in distribution should be given by $-\cG_u \nu_\eps u$, where $\cG_u$ is the solution operator associated to the linearization of the homogenized equation around the homogenized solution $u$. In this paper, we derive an expansion formula for $u^\eps - u$ (see \eqref{e.expan} below), which plays the role of \eqref{e.lexpan}. In fact, $\cG$ above is replaced by $\cG_u$ and there are two more terms involving the nonlinearity $f$. To overcome the difficulty caused by the nonlinearity, we first establish uniform $H^1$ and $L^\infty$ estimates on solutions to \eqref{e.rpde} that are independent of $\eps$ and $\om$. These uniform bounds, together with regularity assumptions on $f$, allow us to estimate $f(u^\eps) - f(u)$ as $u^\eps - u$ and estimate $f(u^\eps) - f(u) - f'(u)(u^\eps - u)$ as $(u^\eps - u)^2$. As a result, we show that the nonlinear terms in the expansion of $u^\eps - u$ do not contribute to the limiting distribution.

Finally, the picture of the size and the limiting distribution of the homogenization error $u^\eps - u$ is as follows; see Theorem \ref{t.sl.23} and Theorem \ref{thm:main3} below. In dimension $d = 2,3$, for short range correlated random potential, $u^\eps - u$ is of order $\eps^{\frac d 2}$ in the $L^2(\Om,L^2(D))$ norm, and the distribution in the space $L^2(D)$ of the normalized error $\eps^{-\frac d 2} (u^\eps - u)$ converges weakly to a Gaussian distribution on $L^2(D)$, with correlation kernel determined by the homogenized solution $u$, the Green's function associated to $\cG_u$, and the integral of the correlation function of $\nu(x,\om)$. In the long range case, for potentials that are functions of long range correlated Gaussian random fields, the size of $u^\eps - u$ is of order $\eps^{\frac \alpha 2}$, where $\alpha \in (0,d)$ is the decaying rate of the correlation of the underlying Gaussian field. The limiting distribution of $\eps^{-\frac \alpha 2}(u^\eps - u)$ is a long range correlated Gaussian, with correlation kernel depending also on $\alpha$ and certain parameters in the model of the random potential. 

\medskip

The rest of this paper is organized as follows: in section \ref{s.prelim} we make precise the main assumptions on the random potential and the nonlinearity in \eqref{e.rpde}, and state the main theorems. In section \ref{s.homog}, we establish well-posedness of \eqref{e.rpde} and more importantly, uniform $H^1$ and $L^\infty$ bounds that are independent of $\eps$ and $\om$. These results facilitate the homogenization proof. Section \ref{s.qerror} and \ref{s.distr} are devoted to the proof of the main results, where we characterize how the homogenization error scales in energy norm, and determine the limiting distribution of the scaled error. We recall the controls on the terms in the expansion formula that are the same as in the linear setting, without detailing the proof, and concentrate on how to apply the uniform $L^\infty$ bound to deal with the nonlinear terms, which are new. Finally, we make some comments in section \ref{s.discussion} and discuss some directions of further studies.


\section{Assumptions and Main Results}
\label{s.prelim}

\subsection{Assumptions}

We first state the main assumptions on the nonlinearity $f$ and the random potential $q(x,\om)$. Let $\lambda_1(-\Delta;D)$, which is henceforth simplified to $\lambda_1(-\Delta)$ or even $\lambda_1$, be the first eigenvalue of the Dirichlet Laplacian operator on $D$, that is,
\begin{equation}
\lambda_1(-\Delta) := \inf \left\{ \int_D \left| Dw\right|^2 dx \,: \, w \in H^1_0(D), \, \int_D w^2 dx = 1\right\}.
\end{equation}
Since $D$ is bounded, we note that $\lambda_1(-\Delta) > 0$. The first set of main assumptions on $q$ and $f$ are:

\begin{itemize}
\item[{[A1]}] $q(x,\omega)$, $x \in \R^d$, is a stationary ergodic random field on $(\Om,\cF,\bP)$. This means, there exist a family of $\bP$-preserving ergodic group of transformations $\{\tau_x : \Om \to \Om \}_{x \in \R^d}$, i.e.
\begin{equation*}
\begin{aligned}
&\bP(\tau_x A) = \bP(A) \text{ for all } x \in \R^d \text{ and } A \in \cF.\\
&\tau_y A = A \text{ for all } y \in \R^d \text{ implies that } \bP(A) \in \{0,1\}, 
\end{aligned}
\end{equation*}
and a random variable $\widetilde q$ such that $q(x,\om) = \widetilde{q}(\tau_x \om)$.

\item[{[A2]}] There exist $M \ge 1$ such that
\begin{equation}
\label{e.qbdd}
\left| \widetilde{q}(\om) \right| \le M, \quad \text{ for all } \om \in \Om.
\end{equation}
Moreover, there exists $\gam > 1$ such that $|f(s)| \le C(1+|s|^\gam)$ for all $s \in \R$, and if $d=3$, we further assume that $\gam < 2d/(d-2) -1 =  5$.

\item[{[A3]}] $f: \R \to \R$ is continuously twice differentiable, and for some $c > 0$,
\begin{equation}
\lambda_1 + \widetilde{q}(\om) + f'(s) \ge c, \quad \text{for all } \om \in \Om \text{ and } s \in \R.
\label{e.fprime}
\end{equation}
\end{itemize}

The assumptions [A2] and [A3] guarantee that for each $\om \in \Om$, the heterogeneous semilinear equation \eqref{e.rpde} admits a unique weak solution. The stationarity and ergodicity of $q(x,\om)$ further makes sure that this problem homogenizes, in the limit $\eps \to 0$, to the effective equation \eqref{e.hpde}. Notice that [A2][A3] still hold if we replace $\widetilde{q}$ by $\ol{q}$; hence, the homogenized problem is also uniquely solvable. 

To get quantitative estimates on the homogenization error $u^\eps - u$, and to find the limiting distribution of this error after proper scaling, we need more information on the random potential $q(x,\om)$, or equivalently the fluctuation
\begin{equation}
\nu(x,\om) := q(x,\om) - \E q = q(x,\om) - \ol{q}.
\label{e.nudef}
\end{equation}
In most of the paper, we assume that $q(x,\om)$ is a short-range correlated random field satisfying certain fourth-order moment estimates. Let $R(x) = \E (\nu(x,\om) \nu(0,\om)) = \E (\nu(x+y,\om) \nu(y,\om))$ be the auto-correlation function of the stationary field $\nu$. The term ``short range correlation'' refers to the condition that the correlation function $R(x)$ is integrable, i.e.
\begin{equation}
\sigma^2 : = \int_{\R^d} R(x) dx < \infty.
\label{e.sigdef}
\end{equation}
Note that by definition, $R(x)$ is a nonnegative definite function in the sense that for any positive integer $n$, for any $n$-tuples $(x_1,\cdots,x_n)$, the matrix formed by $(R(x_i - x_j))_{i,j = 1,\cdots,n}$, is nonnegative definite. By Bochner's theorem, the integral of $R$, i.e. $\sigma^2$, is nonnegative. Throughout the paper, we assume also that $\sigma^2 > 0$. 

We will need an estimate for (mixed) fourth order moments of $\nu$ later. To simplify the presentation, we impose a stronger condition using the notion of ``maximal correlation coefficient", which quantifies how fast the correlation of $\nu$ decays. Let $\mathcal{C}$ the set of compact sets in $\R^d$, and for two sets $K_1, K_2$ in $\mathcal{C}$, the distance $d(K_1,K_2)$ is defined to be
$$
d(K_1,K_2) = \min_{x \in K_1, y \in K_2} \, |x - y|.
$$
Given any compact set $K \subset \mathcal{C}$, we denote by $\cF_K$ the $\sigma$-algebra generated by the random variables $\{\nu(x) \,:\, x \in K\}$. The maximal correlation coefficient $\varrho$ of $\nu$ is defined as follows: for each $r > 0$, $\varrho(r)$ is the smallest value such that the bound
\begin{equation}\label{e.varrho}
\E \left( \varphi_1(\nu) \varphi_2(\nu)\right) \le \varrho(r) \sqrt{\E \left(\varphi^2_1(\nu)\right) \, \E \left(\varphi^2_2(\nu) \right)}
\end{equation}
holds for any two compact sets $K_1, K_2 \in \mathcal{C}$ such that $d(K_1,K_2) \ge r$ and for any two random variables of the form $\varphi_i(\nu)$, $i=1,2$, such that $\varphi_i(\nu)$ is $\cF_{K_i}$-measurable and $\E \varphi_i(\nu) = 0$. 

\medskip

The further assumption we have on $q(x,\om)$ is:

\begin{enumerate}
\item[{[S]}] The maximal correlation function satisfies $\varrho^{\frac 1 2} \in L^1(\R_+,r^{d-1} dr)$, that is
$$
\int_0^\infty \varrho^{\frac 1 2}(r) r^{d-1} dr < \infty.
$$
\end{enumerate}

Assumptions on the mixing coefficient $\varrho$ of random media have been used in \cite{B-CLH-08, BJ-CMS-11,HPP13}; we refer to these papers for explicit examples of random fields satisfying the assumptions. Note that the correlation function $R(x)$ of $\nu$ can be bounded by $\varrho$. For any $x \in \R^d$,
\begin{equation*}
|R(x)| = |\E \nu(x,\om) \nu (0,\om)| \le \varrho(|x|) \mathrm{Var} (q).
\end{equation*}
By [A2][A3], $q$, and hence its variance are bounded. Since we may assume $\varrho \in [0,1]$ (hence $\varrho \le \sqrt{\varrho}$), [S] implies $R \in L^1(\R^d)$ and $q(x,\omega)$ has short range correlations. In fact, $|R|^{\frac 1 2} \in L^1(\R^d)$, so roughly speaking, $|R(x)| \lesssim 1/|x|^\alpha$ with $\alpha > 2d$. Therefore, [S] is a stronger condition than $q$ being short range correlated, which would just require the decay rate $\alpha > d$. This faster decay of correlation is not necessary for the main results of this paper to hold (see Remark \ref{r.hypo} below). But using the assumption [S], we can simplify significantly the following fourth order moment estimates of the random potential $\nu(x,\omega)$.

For any four points $x,y,t$ and $s$ in $\R^d$, define
\begin{equation}\label{e.Psi}
\Psi_\nu(x,y,t,s) := \E\,\nu(x)\nu(y)\nu(t)\nu(s) - (\E\,\nu(x)\nu(y)) (\E\,\nu(t)\nu(s)).
\end{equation}
Were $\nu$ a Gaussian random field, its fourth order moments would decompose as a sum of  products of pairs of $R$ and the above quantity simplifies to a sum of two products of correlation functions. This property does not hold for general random fields. However, we have the similar estimate:
\begin{equation}
\label{e.cumu}
|\Psi_\nu(x,t,y,s)| \le \vartheta(|x-t|)\vartheta(|y-s|) + \vartheta(|x-s|)\vartheta(|y-t|)
\end{equation}
where $\vartheta(r) = (K\varrho(r/3))^{\frac 1 2}$, $K = 4\|\nu\|_{L^\infty(\Omega\times D)}$. We refer to \cite{HPP13} for the proof of this lemma. Estimates of this type based on mixing property already appeared in \cite{B-CLH-08}. 

\bigskip

\noindent{\bf Notations}. 
We simplify the notation of Lebesgue spaces $L^p(D)$ on $D$ to $L^p$ when this is not confusing; in particular, if Lebesgue space on the probability space $\Om$ is concerned, we make the dependence explicit. We use $H^1$ for the Sobolev space $W^{1,2}$ and we use $H^s$, $s \in (0,1)$, for the fractional Sobolev space $W^{s,2}$, which is the closure of $C^\infty_0(D)$ in the norm
\begin{equation*}
\|u\|_{H^s(K)}^2 := \|u\|^2_{L^2(K)} + \int_{K^2} \frac{|u(x) - u(y)|^2}{|x-y|^{d+2s}} dx dy.
\end{equation*}
See \cite{DPV12} for more reference on $H^s$. Throughout the paper, $C$ denotes various bounding constants that may change line after line and we say $C$ is universal when it depends only on the parameters in the main assumptions above. For the functions $q, \nu$ and $R$, we use subscript $\eps$ to denote the scaled function, as in $q_\eps(x) = q\left(\frac{x}{\eps}\right)$.

\subsection{Main results}

The first main theorem concerns how the homogenization error scales.

\begin{theorem} \label{t.size}
Let $D \subset \R^d$ be an open bounded domain with $C^2$ boundary $\partial D$, $u^\eps$ and $u$ be the solutions to \eqref{e.rpde} and \eqref{e.hpde} respectively. Suppose that {\upshape[A1][A2][A3]} and {\upshape[S]} hold, $g \in L^2(D)$ and $d = 2,3$. Then, there exists positive constant $C$, depending only on the universal parameters and $\|g\|_{L^2}$, such that
\begin{equation}\label{e.t.size}
\E\,\|u^\eps - u\|^2_{L^2} \le C(\|g\|_{L^2}) \eps^d.
\end{equation}
\end{theorem}

This theorem provides $L^2(\Omega,L^2(D))$ estimate of $u^\eps - u$, and its proof is detailed in section \ref{s.qerror}. So roughly speaking, the size of the homogenization error is of order $\eps^{\frac d 2}$. We hence consider the limiting law of the rescaled random fluctuation $\eps^{-\frac{d}{2}}(u^\eps - u)$, and prove:

\begin{theorem}
\label{t.sl.23}
Suppose that the assumptions in Theorem \ref{t.size} hold. Let $\sigma$ be defined as in \eqref{e.sigdef} and $G_u(x,y)$ be the Green's function defined by the linearized equation around $u$; see \eqref{e.Green} below. Let $W(y)$ denote the standard $d$-parameter Wiener process. Then as $\eps \to 0$, 
\begin{equation}
\label{e.sl.23}
\frac{u^\eps - u}{\sqrt{\eps^d}} \xrightarrow{\rm{distribution}} \sigma \int_D G_u (x,y) u(y) dW(y), \quad \text{ in  } L^2(D).
\end{equation}
\end{theorem}

The proof can be found towards the end of section \ref{s.distr}.

\begin{remark}\label{r.Gaussian}
The integral on the right hand side of \eqref{e.sl.23} is understood, for each fixed $x$, as Wiener integral in $y$ with respect to the multiparameter Wiener process $W(y)$. Let $X$ denote this integral. For $d = 2,3$, because the Green's function $G_u(x,y)$ is square integrable (see below), $X$ is a random element in $L^2(D)$. In particular, for any $\varphi \in L^2(D)$, the inner product $\langle \varphi, X\rangle$ has precisely the Gaussian distribution $\mathcal{N}(0,\sigma_\varphi^2)$ with mean zero and variance
\begin{equation}
\sigma_\varphi^2 := \sigma^2 \int_D (\cG_u \varphi)^2 u^2 dy.
\end{equation}
\end{remark}

\begin{remark}\label{r.hypo} We make some comments on the main assumptions of this paper. In view of the variational formulation of \eqref{e.rpde}, assumption [A3] is natural because it guarantees that the minimizing functional is coercive and convex, as we show in the next section. The full zero order term $f^\eps(u) : = q_\eps u + f(u)$ can be viewed as a reaction-diffusion nonlinearity. As an example, when $f$ is of bistable type, say $f(u) = u(u-1)(u-\theta)$ for some $\theta \in (0,1)$, and $q$ is large enough, all of the requirements in [A3] are satisfied.

As pointed out earlier, the assumption [S] imposes stronger decay rate on the correlation function $R$ than it being integrable. This stronger condition is mainly imposed to simplify the control of the fourth order moment $\Psi_\nu$ in \eqref{e.cumu}. We refer to \cite{BJ-CMS-11} for an alternative way to control terms like $\Psi_\nu$, which allows decay rate $|R(x)| \lesssim |x|^{-\alpha}$, $\alpha > d$, at the cost that there are more pairs of products of controlling functions on the right hand side of \eqref{e.cumu}.

We state and prove the main theorems assuming that the random potential $q(x,\om)$ has short range correlation. Nevertheless, our approach also works for some long range correlation setting; see the last section of this paper.
\end{remark}

\section{Homogenization and Uniform Estimates}
\label{s.homog}

\subsection{Well-posedness and uniform estimates}

We record here first the well-posedness result for the heterogeneous semilinear equation \eqref{e.rpde}.

\begin{lemma}\label{l.well}
Under the assumptions [A2][A3], for each fixed $\om\in \Om$ and $\eps \in (0,1)$, $g \in H^{-1}$, there exists a unique weak solution in $H^1_0(D)$ to \eqref{e.rpde}. Moreover, there exists $C > 0$ independent of $\eps$ and $\om$, and
\begin{equation}
\label{e.l.well}
\|u^\eps\|_{H^1_0} \le C(\|g\|_{H^{-1}}).
\end{equation}
\end{lemma}

This result is classical; nonetheless, we briefly outline the proof for the sake of completeness.

\begin{proof}
For fixed $\om$ and $\eps$, one can recast \eqref{e.rpde} as the Euler-Lagrange equation associated to the variation problem of minimizing the following (nonlinear) energy functional 
\begin{equation}
\label{e.rvar}
I^\eps[v] :=  \int_D \frac{1}{2} |\nabla v|^2 + \frac{1}{2} q\left(\frac{x}{\eps},\om\right) v^2 + F(v) - v g \ dx,
\end{equation}
over the set $H^1_0(D)$, where $F$ is an antiderivative of $f$, e.g. $F(v) = \int_0^v f(s) ds$. 

The assumption $|f(u)| \le C(1+|u|^\gam)$ implies that $|F(u)| \le C(1+|u|^{\gam + 1})$. By Sobolev embedding, if $d = 2$, $H^1(D)$ is embedded in $L^p(D)$ for any $p \in [1,\infty)$ and if $d = 3$, $H^1(D)$ is embedded in $L^6(D)$ which is included in $L^{\gam + 1}$ since $\gam \le 5$. In both cases, $F(v)$ is integrable. So, the integral above is well defined on $H^1_0(D)$. 

The assumption $\lambda_1 + \widetilde{q} + f'(s) \ge c > 0$ for all $s \in \R$ further implies that $I^\eps$, as a functional on $H^1_0(D)$, is strictly convex. Moreover, we can choose $\alpha > 0$ small so that $(1-\alpha)\lambda_1 + \widetilde{q} + f'(s) \ge \frac{c}{2} > 0$ still holds for all $s$ and $\om$. Integrate in $s$ on both sides of this inequality; we get
\begin{equation*}
\frac{1}{2} \left( (1-\alpha) \lambda_1 + q_\eps(x,\om) \right) s^2 + F(s) \ge \frac{c}{4}s^2 + f(0) s.
\end{equation*}
Using this fact and the Poincar\'e inequality $\|\nabla v\|_{L^2}^2 \ge \lambda_1 \|v\|_{L^2}^2$, we get
\begin{equation}
\label{e.coercive}
\begin{aligned}
\int_D \frac{1}{2}|\nabla v|^2 + \frac{1}{2} q_\eps(x,\om) v^2 + F(v) \, dx &\ge \int_D \frac{\alpha}{2}|\nabla v|^2 + \frac{c}{4} |v|^2 \, dx  + \int_D f(0) v \, dx\\
&\ge \int_D \frac{\alpha}{2}|\nabla v|^2 + \frac{c}{8} |v|^2\, dx - \frac{2|f(0)|^2}{c}|D|,
\end{aligned}
\end{equation}
where $|D|$ is the volume of the domain $D$. This shows that $I^\eps$ is also bounded from below. As a result, \eqref{e.rvar} and hence \eqref{e.rpde} admits a unique solution for each $\eps$ and $\om$.

For the uniform $H^1$ bound, we solve, any $g \in H^{-1}(D)$, the deterministic Poisson problem
\begin{equation*}
-\Delta u_g = 2g, \quad \text{ in } D,
\end{equation*}
with boundary condition $u_g = 0$ on $\partial D$. Clearly we have $\|u_g\|_{H^1} \le C\|g\|_{H^{-1}}$ for some universal constant $C$. The weak formulation of $u_g$ then yields,
\begin{equation*}
I^\eps[u_g] =\int_D \frac{1}{2} |\nabla u_g|^2 + \frac{1}{2} q\left(\frac{x}{\eps},\om\right) (u_g)^2 + F(u_g) - g u_g \ dx = \frac{1}{2} \int_D q_\eps(x,\om) u^2_g dx + \int_D F(u_g) dx. 
\end{equation*}
Now we compare $I^\eps[u_g]$ and $I^\eps[u^\eps]$. Note that \eqref{e.coercive} implies 
\begin{equation*}
I^\eps[u^\eps] \ge \frac{1}{2} \alpha \int_D |\nabla u^\eps|^2 \, dx -\int_D g u^\eps - C.
\end{equation*}
Using $I^\eps[u^\eps] \le I^\eps[u_g]$ and then H\"older's and Young's inequalities, we have
\begin{equation*}
\frac{1}{2} \alpha \int_D |\nabla u^\eps|^2 \le \frac{1}{2} \int_D q_\eps u^2_g dx + \int_D F(u_g) dx + 4C(c',\lambda_1) \|g\|_{H^{-1}}^2 + \frac{1}{4} \alpha \int_{D} |\nabla u^\eps|^2 + C,
\end{equation*}
for some big constant $C$. By [A2], the integral of $q_\eps u^2_g$ is bounded by $C\|u_g\|_{L^2}^2$, which is further dominated by $C\|g\|_{H^{-1}}^2$. In view of the growth condition on $F$ and the Sobolev embedding as before, we get $\|F(u_g)\|_{L^1} \le C(\|u_g\|_{H^1}) = C(\|g\|_{H^{-1}})$ for some $C$ depending also on the parameters in the main assumptions but is uniform in $\eps$ and $\om$. The desired result then follows.
\end{proof}

Next, we show that the solution $u^\eps \in L^\infty(D)$, and the bound on $\|u^\eps\|_{L^\infty}$ can be made independent of $\eps$ or $\om$. 

\begin{lemma}\label{l.Linf} Assume that [A2][A3] hold. For any fixed $\eps \in (0,1)$ and $\om \in \Om$, let $u^\eps(\cdot,\om)$ be the unique solution of \eqref{e.rpde}. Then for any $g \in L^2(D)$, there exists $C>0$ independent of $\eps$ and $\om$ such that
\begin{equation}
\|u^\eps\|_{L^\infty} \le C(c,M,\|g\|_{L^2}).
\label{e.Linf}
\end{equation}
\end{lemma}

\begin{proof} We use the following results from elliptic regularity theory: 

(i) Suppose that $v$ solves the linear equation
\begin{equation*}
-\Delta v = h \quad \text{ in } D,
\end{equation*} 
with Dirichlet boundary condition $v = 0$ on $\partial D$, $h \in L^p$ with $p > \max \{1,\frac{d}{2}\}$ and $v \in L^r$ for some $r \in [1,\infty)$, then it holds that $\|v\|_{L^\infty} \le C(\|h\|_{L^p}+\|v\|_{L^r})$ with $C$ also depending on $r$, $p$ and $D$. We refer to \cite[Theorem 4.2.2]{Cazenave_SLE} for the proof of this result.

(ii) Suppose $v$ satisfies $-\Delta v + v = h$ in $D$ with Dirichlet boundary condition on $\partial D$, $h \in H^{-1}(D)$ and $h \in L^p(D)$ for some $p \in (1,\frac{d}{2})$ and $d = 3$. Then there exists $C(p,d,D)$ such that $\|v\|_{L^{\frac{dp}{d-2p}}} \le C\|h\|_{L^p}$. We refer to \cite[Theorem 4.2.3]{Cazenave_SLE} for the proof of this result.

To prove Lemma \ref{l.Linf}, first consider the simple situation of $d = 2$ and $\gamma > 1$ arbitrary in [A2], or $d = 3$ and $1 < \gamma \le 4/(d-2) = 4$. For notational simplicity, we write $u$ for $u^\eps$. Recast the semilinear equation \eqref{e.rpde} into the form of $-\Delta u = h$ with $h = g - f(u) - q_\eps(x,\om) u$. Now assumption [A2] imposes that $\|q_\eps u\|_{L^p} \le C\|u\|_{L^p}$ uniformly in $p$, $\eps$ and $\om$. If $d = 2$, then by Sobolev embedding and the growth condition on $f$, we verify that $h \in L^p$ for all $p \in [1,\infty)$. If $d = 3$ and $\gamma \le 4$, then by Sobolev embedding, $u \in L^{r}$ with $r = \frac{2d}{d-2}$ and $r > \frac{d\gamma}{2}$. Then $f(u) \in L^{\frac r \gamma}$ and $\frac{r}{\gamma} > \frac{d}{2}$. Note also that $g \in L^2$, $q_\eps u \in L^2$ and $2 > \frac{d}{2}$ for $d = 2,3$; we hence verify that $h \in L^p$ for $p > \frac{d}{2}$. Applying regularity theory (i), we conclude that
\begin{equation*}
\|u\|_{L^\infty} \le C(c,M,\|g\|_{L^2},\|u\|_{H^1}) \le C(c,M,\|g\|_{L^2}).
\end{equation*}

For the more general case, $d = 3$ and $\frac{2d}{d-2} \le \frac{d\gam}{2}$ (i.e. $\gam \in (4,5)$), we need a bootstrap argument. Here, we mimic the proof of Theorem 4.4.1 in Cazenave \cite{Cazenave_SLE}. The semilinear equation \eqref{e.rpde} can also be recast as $-\Delta u + u = b$ with $b(x) = g(x) + (1-q_\eps(x,\om)) u - f(u)$. Again the function $g$ and $(1-q_\eps) u$ are in $L^p$ with $p> \frac{d}{2}$ with uniform bounds, so we only need to take care of the nonlinear term $f(u)$. To start, we set $r = \frac{2d}{d-2}$ and we have $u \in L^r$ and we note
\begin{equation*}
\gamma < r < \frac{d \gamma}{2}, \quad \gamma - \frac{2r}{d} = \gamma - \frac{4}{d-2} < \frac{d+2}{d-2} - \frac{4}{d-2} = 1,
\end{equation*}
where we used $\gamma < \frac{d+2}{d-2}$ as in [A2]. In particular, $\theta := \frac{d}{d\gamma - 2r} > 1$. Since $u \in L^r$ and $|f(u)| \le C(1+|u|^\gamma)$, we check that $b \in L^{\frac r \gamma}$ and $\frac{r}{\gamma} \in (1,\frac{d}{2})$. By elliptic regularity (ii), we obtain
\begin{equation*}
\|u\|_{L^{\theta r}} \le C\|b\|_{L^\frac{r}{\gamma}}, \quad \text{since} \quad \theta r = \frac{dr}{d\gamma - 2r} = \frac{d\frac{r}{\gamma}}{d-\frac{2r}{\gamma}}.
\end{equation*}
If $\theta r > \frac{d\gamma}{2}$, we are back to the simple situation. If not, we repeat the argument using elliptic regularity (ii) $k$ more times until $\theta^k r \le \frac{d\gamma}{2} < \theta^{k+1} r$. We then get $u \in L^{\theta^{k+1}r}$ with $\theta^{k+1}r > \frac{d\gamma}{2}$ and hence return to the simple situation, and we can conclude. We further verify that the bounds are uniform in $\eps$ and $\om$ in view of \eqref{e.l.well}.
\end{proof}
\subsection{Homogenization theory}

Because the random coefficients appear only in the zeroth order linear term, i.e. the potential term, the homogenization theory for the equations is relatively straightforward. For the sake of completeness, we present the details here.

\begin{theorem}\label{t.homog} Under assumptions [A1][A2] and [A3], there exists an event $\Om_1 \in \cF$ with full probability measure, such that for all $\om \in \Om_1$, $u^\eps(\cdot,\om)$ converges strongly in $L^2(D)$ and weakly in $H^1_0(D)$, to the deterministic function $u \in H^1_0(D)$ that solves \eqref{e.hpde}
\end{theorem}

\begin{proof} Owing to Lemma \ref{l.well} and Lemma \ref{l.Linf}, we note there exists $C > 0$ uniform in $\eps$ and $\om$, such that $\|u^\eps\|_{H^1(D)} + \|u^\eps\|_{L^\infty(D)} \le C$. Also, by ergodic theorem (see. e.g. \cite[section 7.1]{Jikov_book}), there exists $\Om_1 \in \cF$ with $\bP(\Om_1) = 1$ such that $q_\eps(x,\om)$ converges weakly in $L^2_{\mathrm{loc}}(\R^d)$ to $\ol{q}$ for all $\om \in \Om_1$. As a consequence, for each fixed $\om \in \Om_1$, we can find a function $u(\cdot,\om) \in H^1_0(D)$ and extract a subsequence $\eps(\om) \to 0$, along which we have
\begin{equation*}
\nabla u^\eps \xrightharpoonup{L^2} \nabla u, \quad u^\eps \xrightarrow{L^2} u, \quad q\left(\frac{\cdot}{\eps},\om\right) \xrightharpoonup{L^2} \ol{q},
\end{equation*}
where $\rightharpoonup$ and $\rightarrow$ denotes, respectively, weak and strong convergence. Owing to the uniform in $\eps$ bound on $u^\eps$ and the regularity assumption $f \in C^2(\R)$,
\begin{equation*}
\left| f(u^\eps) - f(u) \right| \le C_1|u^\eps - u|.
\end{equation*}
Here, $C_1$ is the Lipchitz constant of $f$ inside the domain $[-2C,2C]$ and $C$ is the uniform bound in \eqref{e.Linf}. This implies that $f(u^\eps) \to f(u)$ in $L^2(D)$. As a result, pass to the limit $\eps \to 0$ in the weak formulation
\begin{equation*}
\int_D \nabla u^\eps \cdot \nabla \varphi + q\left(\frac{x}{\eps},\om\right) u^\eps \varphi + f(u^\eps) \varphi - g\varphi \ dx = 0, \quad\quad \text{for all } \varphi \in C^\infty_0(D),
\end{equation*}
we get
\begin{equation*}
\int_D \nabla u \cdot \nabla \varphi + \ol{q} u \varphi + f(u) \varphi - g\varphi \ dx = 0, \quad\quad \text{for all } \varphi \in C^\infty_0(D).
\end{equation*}
This shows that the function $u(\cdot,\om)$ solves the equation \eqref{e.hpde}. Note that this problem has a unique deterministic solution, and hence $u(x,\om) = u(x)$ is independent of $\om$. As a result, for all $\om \in \Om_1$ and along the full sequence $\eps \to 0$, the convergence $u^\eps \to u$ hold. This completes the proof of the theorem.
\end{proof}

\subsection{Green's function estimates for the linearized equation}

Let $u$ be the homogenized solution. The linearized differential operator $\cL_u$, around $u$, of the homogenized semilinear operator $\cL u= -\Delta u+ \ol{q}u + f(u)$ is given by
\begin{equation}
\cL_u (v) = -\Delta v + \left( \ol{q} + f'(u) \right) v
\end{equation}
The Green's function $G_u(x,y)$, $x, y \in D$ and $x \ne y$, associated to $\cL_u$ satisfies
\begin{equation}
\label{e.Green}
\left\{
\begin{aligned}
&-\Delta G_u(x;y) + \left(\ol{q} + f'(u(x))\right) G_u(x;y) = \delta_y, \quad &\quad &\text{ for } x \in D,\\
&G_u(x;y) = 0, \quad &\quad &\text{ for } x \in \partial D.
\end{aligned}
\right. 
\end{equation}
We have the following estimates:

\begin{lemma} Assume [A2][A3], and assume that $D \subset \R^d$ is an open bounded domain with $C^2$ boundary. Then there exists $C > 0$ such that,
\begin{equation}
\label{e.Gbdd}
\left| G_u(x,y) \right| \le
\begin{cases}
\displaystyle \frac{C}{|x-y|^{d-2}} &\text{ for } d = 3\\
\displaystyle C \left(\left|\log|x-y|\right| + 1\right) &\text{ for } d = 2
\end{cases}
\quad \text{ and } \quad \left| \nabla_x G_u(x,y) \right| \le \frac{C}{|x-y|^{d-1}}.
\end{equation}
\end{lemma}

We note that by \eqref{e.Linf}, the potential function $\ol{q} + f'(u) \in L^\infty(D)$. Moreover, [A3] guarantees that the problem remains elliptic. The first bound in \eqref{e.Gbdd} is immediate and the second one follows, say, from standard H\"older regularity for gradients. Note that this is the place where regularity of $\partial D$ is used.


\section{Quantitative Estimates of the Homogenization Error}
\label{s.qerror}

In this section, we determine the convergence rate of $u^\eps \to u$ in $L^2(\Om,L^2(D))$. Define  $\xi^\eps:=u^\eps - u$. Then, it satisfies
\begin{equation*}
-\Delta \xi^\eps + q_\eps \xi^\eps + f(u^\eps) - f(u) = -\nu_\eps(x,\om) u.
\end{equation*}
Formally, the nonlinear term $f(u^\eps) - f(u)$ is approximated by $f'(u)(u^\eps - u)$. So we expect that the leading term in $\xi^\eps$ is given by the solution of the following equation:
\begin{equation}
\left\{
\begin{aligned}
&\cL_u \chi^\eps = -\Delta \chi^\eps + \ol{q} \chi^\eps + f'(u) \chi^\eps = -\nu_\eps u, &\quad &\quad \text{ in } D,\\
&\chi^\eps = 0, &\quad &\quad\text{ on } \partial D.
\end{aligned}
\right.
\end{equation}
Let $\cG_u$ to be the inverse operator, i.e. the fundamental solution operator of the Dirichlet problem, and $G_u(x,y)$ is the Green's function. In view of [A3], these notations are well defined, and we have
\begin{equation}
\label{e.chi}
\chi^\eps = -\cG_u \nu_\eps u.
\end{equation}

Our goal is to estimate the quantity $\|\xi^\eps\|_{L^2(\Om,L^2(D))}$. First, an estimate for $\chi^\eps$ in $L^2(\Om,L^2(D))$ is easily obtained from the linear theory. Next, by applying the mean value theorem to the nonlinear term in the equation of $\xi^\eps$, we show that the remainder $z^\eps := \xi^\eps - \chi^\eps$ satisfies a linear equation with coefficients that depend on $u^\eps$ and $u$, and with $\chi^\eps$ on the right hand side. Therefore, by linear theory again, we finally obtain estimates for $z^\eps$ and hence for $\xi^\eps$. To use the linear theory for $z^\eps$, however, we need the uniform (in $\eps$ and $\om$) bound on $u^\eps$ because the coefficient of the equation for $z^\eps$ depends on $u^\eps$.

\medskip

We first present the estimates for the corrector $\chi^\eps$ and briefly recall the proof. For detailed argument, we refer to \cite[Lemma 4.1]{Jing15}.

\begin{lemma}\label{l.chi} Let $d = 2,3$. Assume that [A1][A2][A3] and [S] hold. Then there exists $C>0$ such that 
\begin{equation}
\E \|\chi^\eps\|_{L^2}^2 \le C\eps^d.
\end{equation}
\end{lemma}

\begin{proof} The mean square of the $L^2(D)$ norm of the function $\chi^\eps(\cdot,\om)$ is given by
\begin{equation*}
\E \|\chi^\eps\|_{L^2(D)}^2 = \E \|\cG_u \nu_\eps u\|_{L^2(D)}^2 = \E \int_{D^3} G_u(x,y) G_u(x,z) \nu_\eps(y) \nu_\eps(z) u(y) u(z) dy dz dx.
\end{equation*}
Note that $\E \nu_\eps(y) \nu_\eps(z) = R^\eps(y-z)$ and use the bounds on the Green's function. We get, for $d = 3$,
\begin{equation*}
\E \|\chi^\eps\|_{L^2(D)}^2 \le \int_{D^3} \frac{C}{|x-y|^{d-2} |x-z|^{d-2}} \left|R\left(\frac{y-z}{\eps}\right)\right| |u(y) u(z)| dy dz.
\end{equation*}
In view of the bound \eqref{e.Gbdd}, the product of $G_u(\cdot,y)$ and $G_u(\cdot,z)$ is integrable on $D$ with bounds independent of $y$ or $z$. Hence, we can integrate over $x$ first and then change variable in the remaining integral, which yields a factor of $\eps^d$ and verifies the desired result. The situation of $d=2$ can be treated in the same manner.
\end{proof}

Next, we move on the remainder $z^\eps = u^\eps - u - \chi^\eps$. In view of the equations satisfied by $u^\eps$, $u$ and $\chi^\eps$, we have
\begin{equation*}
(-\Delta + q_\eps) z^\eps + f(u^\eps) - f(u) = -\nu_\eps \chi^\eps + f'(u) \chi^\eps.
\end{equation*}
Let $h_\eps$ be the function
\begin{equation*}
\frac{f(u^\eps)-f(u)}{u^\eps - u} \mathbf{1}_{\{|u^\eps-u|>0\}}
\end{equation*}
where $\mathbf{1}$ denotes the indicator function. Then, in view of $f\in C^1(\R)$ and the bound \eqref{e.Linf}, we have $h_\eps \in L^\infty(D)$. We verify that $z^\eps$ solves the Dirichlet problem
\begin{equation}
-\Delta z^\eps + q_\eps z^\eps + h_\eps z^\eps = \nu_\eps \chi^\eps + (f'(u) - h_\eps) \chi^\eps,
\end{equation}
with zero boundary condition. Indeed, this problem has unique solution and we verify easily that $u^\eps - u - \chi^\eps$ is the solution. We have the following result.

\begin{lemma}\label{l.z}
Assume that the conditions in Lemma \ref{l.chi} hold. Then there exists $C>0$ such that
\begin{equation}
\E \|z^\eps\|_{L^2}^2 \le C\eps^d.
\end{equation}
\end{lemma}

\begin{proof} Let $\cL^{\eps,\om}$ denote the random linear differential operator $-\Delta +  (q_\eps + h_\eps)$. In view of [A3], the uniform bound on $u^\eps$ and the definition of $h_\eps$, we observe that for each fixed $\eps \in (0,1)$ and $\om \in \Om$, the potential term $p^\eps: = q_\eps + h_\eps$ satisfies
\begin{equation*}
-\lambda_1 + c \le p^\eps \le M'.
\end{equation*}
By standard elliptic theory, $\cL^{\eps,\om}: H^1_0 \to H^{-1}$ is invertible. In particular the inverse operator $(\cL^{\eps,\om})^{-1}$ is also a bounded transformation on $L^2(D)$. In fact, $\|(\cL^{\eps,\om})^{-1}\|_{L^2 \to L^2} \le c^{-1}$. Applying this fact, we have
\begin{equation*}
\|z^\eps\|_{L^2}^2 \le 2c^{-2} \left( \|\nu_\eps \chi^\eps\|_{L^2}^2 + \|(f'(u) - h_\eps)\chi^\eps\|^2_{L^2} \right).
\end{equation*}
By [A2] and [A3], $\|\nu_\eps\|_{L^\infty} \le C$ uniformly in $\eps$ and $\om$. By the uniform bounds on $u^\eps, u$, and by the smoothness of $f$, we confirm that $\|f'(u) - h_\eps\|_{L^\infty} \le C$ also uniformly in $\eps$ and $\om$. As a result, we have $\|z^\eps\|_{L^2}^2 \le C\|\chi^\eps\|^2_{L^2}$. The desired result follows.
\end{proof}

\medskip

The following control of $u^\eps - u$ follows immediately.

\begin{corollary}\label{c.xi} Under the same conditions of Lemma \ref{l.chi}, there exists $C>0$ such that 
\begin{equation}\label{e.c.xi}
\E \|u^\eps - u\|_{L^2}^2 \le C \eps^{d}.
\end{equation}
\end{corollary}

\bigskip

\section{Limiting Distributions of the Homogenization Error}
\label{s.distr}

\medskip

In this section, we identify the limiting distribution of the scaled homogenization error, i.e. the limit of the law of $\eps^{-\frac d 2} (u^\eps - u)$ in the space of $L^2(D)$ functions.

\subsection{Expansion formula and overview of proof}

To characterize the limiting distribution of $\eps^{-\frac d 2}(u^\eps - u)$, we extend to nonlinear equations the framework developed in \cite{B-CLH-08,BJ-CMS-11,Jing15} for linear equations. The first step is to find a suitable expansion formula similar to \eqref{e.lexpan} for the homogenization error.

We observe that $z^\eps = u^\eps - u - \chi^\eps$ satisfies the linear equation
\begin{equation*}
-\Delta z^\eps + \ol{q} z^\eps + f'(u) z^\eps = -\nu_\eps \xi^\eps - (f(u^\eps) - f(u) - f'(u)\xi^\eps).
\end{equation*}
Hence $z^\eps$ has the following representation
\begin{equation}
z^\eps = - \cG_u \nu_\eps \xi^\eps - \cG_u \left( f(u^\eps) - f(u) - f'(u) \xi^\eps \right).
\end{equation}
Substituting this formula to $u^\eps - u = \chi^\eps + z^\eps$ and using the formula $\chi^\eps = -\cG_u \nu_\eps u$, we obtain
\begin{equation}
\label{e.expan}
\begin{aligned}
u^\eps - u =& -\cG_u \nu_\eps u - \cG_u \nu_\eps (u^\eps - u) - \cG_u (f(u^\eps) - f(u) - f'(u)\xi^\eps)\\
=&  -\cG_u \nu_\eps u + \cG_u \nu_\eps \cG_u \nu_\eps u +  \cG_u \nu_\eps \cG_u \nu_\eps (u^\eps - u) \\
 & - \cG_u (f(u^\eps) - f(u) - f'(u)\xi^\eps) + \cG_u \nu_\eps \cG_u (f(u^\eps)-f(u) - f'(u)\xi^\eps).
\end{aligned}
\end{equation}

With this expansion formula at hand, we will find the limiting distribution of $\eps^{-\frac d 2}(u^\eps - u)$ by examining the terms on the right hand side one by one. In particular, we note that the first three terms are  the same as those obtained  for linear equations, while the last two terms involve the nonlinearity.

We first recall the standard criterion for establishing weak convergence of probability measures determined by random processes that are $L^2(D)$ functions.  

\begin{theorem}\label{t.criterion}
Let $\{X^\eps\}$, $\eps \in (0,1)$, be a family of random processes on the probability space $(\Om,\cF,\bP)$ and $\{X^\eps\} \subset L^2(D)$. Then $X^\eps$ converges in distribution in $L^2(D)$ to the random process $X$ in $L^2(D)$ if
\begin{itemize}
\item[(i)] \emph{(Convergence of inner products with test functions)}. For any $\varphi \in L^2(D)$, the random variable $\langle \varphi, X^\eps \rangle$ converges in distribution in $\R$ to $\langle \varphi, X \rangle$.
\item[(ii)] \emph{(Tightness of distributions)}. The family of distributions in $L^2(D)$ determined by $\{X^\eps\}$ is tight.
\end{itemize}
\end{theorem}

This Prohorov type result is well known and we refer to {\cite[Chapter VI, Lemma 2.1]{Partha}} for a proof in the general Hilbert space setting. When the tightness of the distribution of $\{X^\eps\}$ is concerned, the following result becomes handy provided that $\{X^\eps\}$ in fact admits higher regularity.

\begin{lemma}\label{l.tight}
Let $\{X^\eps\}$ be as in Theorem \ref{t.criterion}. Suppose further that $\{X^\eps\} \subset H^s(D)$, for some $s \in (0,1)$. Then the family of distributions in $L^2(D)$ determined by $\{X^\eps\}$ is tight if there exists some constant $C > 0$ such that
\begin{equation}\label{e.tight}
\E \|X^\eps\|_{H^s(D)} \le C.
\end{equation}
\end{lemma}

We refer to \cite[Theorem A.1]{Jing15} for a proof. Since $X^\eps$ is $\eps^{-\frac d 2}(u^\eps - u)$ in this paper, it indeed belongs to the more regular space $H^1(D)$. Nevertheless, to obtain the control \eqref{e.tight}, which is uniform in $\eps$, requires some work.

\subsection{Liming distribution for the homogenization error}

The first three terms in the expansion formula \eqref{e.expan} are precisely those encountered in the linear setting; compare with \eqref{e.lexpan}. We recall the following results.

\begin{lemma}\label{l.lin} Assume that the conditions of Theorem \ref{t.size} hold. Then as $\eps \to 0$, we have
\begin{equation}
- \frac{\cG_u \nu_\eps u}{\sqrt{\eps^d}} \xrightarrow{\mathrm{distribution}} \sigma \int_D G_u(x,y) u(y) dW_y,
\end{equation}
with convergence in distribution in the space $L^2(D)$. Moreover, $\cG_u \nu_\eps \cG_u \nu_\eps (u^\eps - u)$ converges in $L^1(\Om,L^2(D))$ and $\cG_u \nu_\eps \cG_u \nu_\eps u$ converges in $L^2(\Om,L^2(D))$ to the zero function.
\end{lemma}

We briefly explain how these results are proved and refer to \cite[Lemmas 4.2, 4.4, 4.5, 5.1 and 5.3]{Jing15} for detailed proofs. 

(1) To show that $- \eps^{-\frac d 2} \cG_u \nu_\eps u$ has the correct limit, according to the criterion Theorem \ref{t.criterion}, Remark \ref{r.Gaussian} and Lemma \ref{l.tight}, it suffices to show, for all $\varphi \in L^2(D)$,
\begin{equation}
\label{e.il1}
 \frac{1}{\sqrt{\eps^d}} \langle -\cG_u \nu_\eps u, \varphi \rangle = \frac{1}{\sqrt{\eps^d}} \int_D \nu\left(\frac{x}{\eps},\om\right) u(x)m(x) dx \xrightarrow[\eps \to 0]{\rm{distribution}} \mathcal{N}\left(0,\sigma^2_\varphi\right),
\end{equation}
where $m := \cG_u \varphi$ and $\sigma^2_\varphi = \sigma^2 \|mu\|^2_{L^2(D)}$, and for some $C>0$ and $s \in (0,1)$ independent of $\eps$,
\begin{equation}
\label{e.Hs1}
\E \left\| \eps^{-\frac d 2} \cG_u \nu_\eps u \right\|_{H^s}^2 \le C.
\end{equation}
We recall that \eqref{e.il1} follows from a generalized central central limit theorem for oscillatory integrals with short range correlated random fields; see \cite[Theorem 3.7]{B-CLH-08}. The uniform $H^s$ estimate \eqref{e.Hs1} for $s$ can be found in \cite[Lemma 5.1]{Jing15}; see Lemma 5.4 below where such a control is needed to estimate the first nonlinear term in \eqref{e.expan}.

\medskip

(2) To show that $\cG_u \nu_\eps \cG_u \nu_\eps (u^\eps - u)$ converges to the zero function in $L^1(\Om,L^2(D))$, we need the fact that for $d = 2,3$,
\begin{equation}
\label{e.GnuG}
\E\|\cG_u \nu_\eps \cG_u\|_{L^2 \to L^2}^2 \le C\eps^d,
\end{equation}
where $\|\cG_u \nu_\eps \cG_u\|_{L^2 \to L^2}$ is the operator norm. This can be proved by using the Green's function bound \eqref{e.Gbdd} and assumption [S]; see \cite[Lemma 4.5]{Jing15}. We then get
\begin{equation*}
\|\cG_u \nu_\eps \cG_u \nu_\eps (u^\eps - u)\|_{L^2} \le \|\cG_u \nu_\eps \cG_u\|_{L^2\to L^2} \|\nu_\eps\|_{L^\infty} \|u^\eps - u\|_{L^2}.
\end{equation*}
Take expectation, apply H\"older inequality and then \eqref{e.Gbdd} and Corollary \ref{c.xi} to get the desired estimate.

(3) To show that $\cG_u \nu_\eps \cG_u \nu_\eps u$ converges to zero in $L^1(\Om,L^2(D))$, the argument above is not valid because $\|\nu_\eps u\|_{L^2}$ is possibly of order one. To circumvent this lack of control, we note that
\begin{equation*}
\|\cG_u \nu_\eps \cG_u \nu_\eps u\|^2_{L^2} \le 2\left( \|\cG_u \nu_\eps \cG_u \nu_\eps u - \E(\cG_u \nu_\eps \cG_u \nu_\eps u)\|^2_{L^2} + \|\E(\cG_u \nu_\eps \cG_u \nu_\eps u)\|_{L^2}^2\right).
\end{equation*}

For the mean function $\E(\cG_u \nu_\eps \cG_u \nu_\eps u)$, we have
\begin{equation*}
\begin{aligned}
& \|\E(\cG_u \nu_\eps \cG_u \nu_\eps u)\|_{L^2}^2 = \int_D \left( \int_{D^2} G_u(x,y) G_u(y,z) R\left(\frac{y-z}{\eps}\right) u(z) dz dy\right)^2 dx\\
= &\int_{D^5} G_u(x,y) G_u(x,\xi) G_u(y,z) G_u(\xi,\eta) R\left(\frac{y-z}{\eps}\right)R\left(\frac{\xi-\eta}{\eps}\right) u(z) u(\eta)d\xi d\eta dy dz dx.
\end{aligned}
\end{equation*}
Integrate over $x$ first and note that, for $d=3$, $\|u\|_{L^\infty} \le C$ and $G_u$ is square integrable and
\begin{equation*}
\int_D |G_u(x,y) G_u(x,\xi)| dx \le C\int_{D} \frac{1}{|x-y|^{d-2} |x-\xi|^{d-2}} dx \le C.
\end{equation*}
We get
\begin{equation*}
\|\E(\cG_u \nu_\eps \cG_u \nu_\eps u)\|_{L^2}^2 \le C \int_{D^4} \frac{1}{|y-z|^{d-2} |\xi-\eta|^{d-2}} \left| R\left(\frac{y-z}{\eps}\right)R\left(\frac{\xi-\eta}{\eps}\right) \right| d\xi d\eta dy dz.
\end{equation*}
After a change of variables and using the fact that $R(\cdot)/|\cdot|^{d-2}$ is integrable over $\R^d$, we verify that the above integral is of order $\eps^{4} \ll \eps^d$. It follows that $\|\eps^{-\frac d 2}\E(\cG_u \nu_\eps \cG_u \nu_\eps u)\|_{L^2}$ converges to zero. The case of $d = 2$ is similar.

For the variation $I_2 := \cG_u \nu_\eps \cG_u \nu_\eps u - \E(\cG_u \nu_\eps \cG_u \nu_\eps u)$, we have for $d = 3$,
\begin{equation*}
\begin{aligned}
\E\|I_2\|_{L^2}^2 & = \E \int_D \left( \int_{D^2} G_u(x,y)G_u(y,z)\left[ \nu_\eps(y)\nu_\eps(z) - \E \nu_\eps(y) \nu_\eps(z)\right] u(z) dz dy\right)^2 dx\\
& = \int_{D^5} G_u(x,y) G_u(x,y') G_u(y,z) G_u(y',z')u(z) u(z')\Psi_{\nu} \left(\frac{y}{\eps},\frac{z}{\eps},\frac{y'}{\eps},\frac{z'}{\eps}\right) dz' dy' dz dy dx.\\
\end{aligned}
\end{equation*}
Now, using the estimates for $\Psi_{\nu}$ given by \eqref{e.cumu}, and by standard techniques for potential integrals, we get $\E\|I_2\|^2_{L^2} \le C\eps^{2d} \ll \eps^d$, which shows that $\eps^{-\frac d 2} I_2$ converges to the zero function in $L^2(\Om,L^2(D))$. For $d=2$, we have the same result.

\medskip

Now we deal with the two nonlinear terms in \eqref{e.expan}.

\begin{lemma} Under the assumptions of Theorem \ref{t.size}, as $\eps \to 0$, we have
\begin{equation}
\frac{\cG_u(f(u^\eps) - f(u) - f'(u)(u^\eps - u))}{\sqrt{\eps^d}} \xrightarrow{\mathrm{distribution}} 0,
\end{equation}
with convergence in distribution in $L^2(D)$ and $0$ denotes the constant zero function. Moreover, the term $\eps^{-\frac d 2}\cG_u \nu_\eps \cG_u (f(u^\eps) - f(u) - f'(u)(u^\eps - u))$ converges to the zero function in $L^1(\Om,L^2(D))$.
\end{lemma}

\begin{proof} \emph{Part I: The convergence of $\cG_u \nu_\eps \cG_u (f(u^\eps) - f(u) - f'(u)(u^\eps - u))$}. This part is simpler and is similar to the control of the remainder term $\cG_u \nu_\eps \cG_u \nu_\eps (u^\eps - u)$ in the linear setting.

\medskip

Firstly, by Taylor's theorem, the regularity of $f$ and the uniform (in $\eps$) bound of $u^\eps$ and $u$, we have
\begin{equation}
\label{e.Taylor}
|f(u^\eps) - f(u) - f'(u)(u^\eps - u)| \le C'|u^\eps - u|^2,
\end{equation}
where $C' = \|f\|_{C^2([-2C,2C])}$ and $C$ is the bound in \eqref{e.Linf}. This shows
\begin{equation*}
\|f(u^\eps) - f(u) - f'(u)(u^\eps - u)\|_{L^2} \le C \|u^\eps - u\|_{L^2}.
\end{equation*}
Applying the uniform bounds on the operator norm of $\cG_u \nu_\eps \cG_u$, we have again
\begin{equation*}
\|\cG_u \nu_\eps \cG_u (f(u^\eps) - f(u) - f'(u)(u^\eps - u))\|_{L^2} \le C\|\cG_u \nu_\eps \cG_u\|_{L^2 \to L^2} \|u^\eps - u\|_{L^2}.
\end{equation*}
The desired result then follows by taking expectation and applying the H\"older inequality, the bound \eqref{e.GnuG} and \eqref{e.c.xi}.
\medskip

\emph{Part II: The limiting distribution of $\eps^{-\frac d 2} \cG_u(f(u^\eps) - f(u) - f'(u)(u^\eps - u))$}. We follow again the criterion in Theorem \ref{t.criterion}.

\medskip

\emph{Convergence in distribution of inner products}. Fix an arbitrary $\varphi \in L^2(D)$. We need to consider the distribution of $\eps^{-\frac d 2} \langle \cG_u (f(u^\eps) - f(u) - f'(u)\xi^\eps, \varphi\rangle = \eps^{-\frac d 2} \langle f(u^\eps) - f(u) - f'(u)\xi^\eps, m\rangle$, where $m = \cG_u \varphi$. Note that $\|m\|_{L^\infty} \le C\|\varphi\|_{L^2}$. Using \eqref{e.Taylor}, we get
\begin{equation*}
\left|  \langle m, f(u^\eps) - f(u) - f'(u)(u^\eps - u)\rangle\right| \le C\|\varphi\|_{L^2} \|u^\eps - u\|_{L^2}^2.
\end{equation*}
As a result, we get
\begin{equation}
\E \left\lvert \eps^{-\frac d 2} \langle m, f(u^\eps) - f(u) - f'(u)(u^\eps - u)\rangle\right\rvert \le C\eps^{-\frac d 2} \E \|u^\eps - u\|_{L^2}^2 \le C\eps^{\frac d 2}.
\end{equation}
This show that $\langle \eps^{\frac d 2} \cG_u (f(u^\eps) - f(u) - f'(u)(u^\eps - u)), \varphi\rangle$ converges, in $L^1(\Om)$ and hence in distribution, to $0$, agreeing with the trivial distribution of $\langle \varphi, 0 \rangle$.

\medskip

\emph{Tightness.} Let $r^\eps = f(u^\eps) - f(u) - f'(u)\xi^\eps$. The function we are considering is $\cG_u r^\eps$. To show tightness of the distributions of $\{\cG_u r^\eps\}$, we control the expectation of the $H^s$ norm of $\cG_u r^\eps$ for some $s \in (0,1)$. For the semi-norm, we calculate
\begin{equation*}
\begin{aligned}
\left[\eps^{-\frac d 2} \cG_u r^\eps\right]_{H^s}^2 &= \frac{1}{\eps^d} \int_{D^2} \frac{\left|\cG_u r^\eps (x) - \cG_u r^\eps(y)\right|^2}{|x-y|^{d+2s}} dy dx \\
&= \frac{1}{\eps^d} \int_{D^4} \frac{ \left(G_u(x,z) - G_u(y,z)\right)  \left(G_u(x,\eta) - G_u(y,\eta)\right)}{|x-y|^{d+2s}} r^\eps(z) r^\eps(\eta)\, dz d\eta dy dx.
\end{aligned}
\end{equation*}
It is proved in \cite[Lemma 5.1]{Jing15} that the uniform bounds \eqref{e.Gbdd} on the Green's function and its gradient imply that
\begin{equation}\label{e.Hs.key}
\int_{D^2} \frac{|(G_u(x,z)-G_u(y,z))(G_u(x,\eta)-G_u(y,\eta))|}{|x-y|^{d+2s}} dy dx \le C,
\end{equation}
for any $s \in (0,\frac{1}{2})$, uniformly for $z, \eta \in D$. We note also that \eqref{e.Taylor} implies, for some $C > 0$ independent of $\eps$ and $\om$,
\begin{equation*}
\|r^\eps\|_{L^1} \le C\|u^\eps - u\|_{L^2}^2.
\end{equation*}
Now we integrate over $x$ and $y$ first and use the inequalities above, and get
\begin{equation*}
\E \left[\eps^{-\frac d 2} \cG_u r^\eps\right]_{H^s}^2 \le C \eps^{-d}\, \E \left(\int_D |r^\eps(z)| dz\right)^2 \le C\eps^{-d}\, \|r^\eps\|_{L^\infty}\E \|r^\eps\|_{L^1} \le C.
\end{equation*}
Similarly, one can show that $\E\|\eps^{-\frac d 2} \cG_u r^\eps\|_{L^2} \le C$ as well. We conclude that $\{\E \|\eps^{-\frac d 2} \cG_u r^\eps\|_{H^s(D)}\}$ is uniformly bounded. It follows that the family of probability distributions in $L^2(D)$ determined by $\{\cG_u r^\eps\}$ is tight. We get the desired result by an application of Theorem \ref{t.criterion}.
\end{proof}

Finally, we combine all of the results above and prove the second main theorem.

\begin{proof}[Proof of Theorem \ref{t.sl.23}] According to the above results, in the expansion formula \eqref{e.expan}, the first term converges in distribution in the space of $L^2(D)$ functions to the desired limit of Theorem \ref{t.sl.23}. All other terms converge in distribution to the zero function, which is deterministic. Therefore, these terms converge to zero also in probability. As a result, the sum of all terms converge in distribution to the limit of the leading term. This concludes our proof.
\end{proof}


\section{Further Discussions}
\label{s.discussion}

\medskip

\subsection{Long range correlated random fields}

We have assumed so far that the random potential $q(x,\om)$ had short-range correlation. Our approach applies to the setting of some long range correlated potentials as well. Following \cite{BGMP-AA-08,BGGJ12_AA}, a large class of long-range correlated potential can be constructed as functionals of  long range correlated Gaussian random fields. Let $q(x,\omega) = \ol{q} + \nu(x,\omega)$ with $\ol{q}$ a nonnegative constant; we assume

\begin{enumerate}
\item[{[L1]}]  $\nu(x,\omega) = \Phi(\frakg(x))$; $\frakg(x,\omega)$ is a centered stationary Gaussian random field 
with unit variance. Furthermore, the correlation function of $\frakg(x,\omega)$ has a heavy 
tail in the sense that, for some positive constant $\kappa_\frakg$ and some real number $\alpha \in (0, 
d)$,
\begin{equation}
R_\frakg(x) := \E\{\frakg(y,\omega)\frakg(y+x,\omega)\} \sim \kappa_\frakg|x|^{-\alpha} \ \mbox{as}\ 
|x|\rightarrow\infty.
\label{e.tail}
\end{equation}

\item[{[L2]}] The function $\Phi: \R \to \R$ 
satisfies $-M \le \Phi + \ol{q} \le M$, and has Hermite rank one, i.e.
\begin{equation}
\int_\R \Phi(s) e^{-\frac{s^2}{2}} ds = 0, \quad \quad V_1 := \int_\R s\Phi(s) e^{-\frac{s^2}{2}} ds \ne 0.
\label{e.odd}
\end{equation}

\item[{[L3]}]The Fourier transform $\hat{\Phi}$ of the function $\Phi$ satisfies
\begin{equation}
\label{e.Phihat}
\int_{\R} |\hat{\Phi}(\xi)| \big( 1 + |\xi|^3 \big) < \infty.
\end{equation}
\end{enumerate}
We henceforth refer assumption [L] to the above conditions all together.

\medskip

The assumption [L2] ensures that $\nu(x,\omega) = \Phi(\frakg(x,\omega))$ is mean zero and the bounds on $\Phi$ ensure that $|q(x,\omega)| \le M$, which is \eqref{e.qbdd}. From the above construction, we check that $\nu(x,\omega)$ is stationary ergodic and has a long-range correlation function  decaying as $\kappa |x|^{-\alpha}$, $\kappa = V_1^2 \kappa_\frakg$; see \cite[Lemma A.3]{Jing15} for the details. Assumption [L3] allows one to derive a (non-asymptotic) estimate, see \cite[Lemma A.5]{Jing15}, for the fourth-order moments of $\nu(x,\omega)$. We have the following analog of Theorem \ref{t.size} and Theorem \ref{t.sl.23}:
\begin{theorem} \label{t.size.l}
Let $d = 2,3$, $u^\eps$ and $u$ be the solutions to \eqref{e.rpde} and \eqref{e.hpde} respectively. Suppose that {\upshape[A1][A2][A3]} and {\upshape[L]} hold, $g \in L^2(D)$. Then, there exists positive universal constant $C$, such that 
\begin{equation}\label{e.t.size.f.l}
\E\,\|u^\eps - u\|_{L^2}
\le 
C(\|g\|_{L^2}) \eps^{\frac \alpha 2}.
\end{equation}
\end{theorem}

Let $W^\alpha(dy)$ be defined formally as $\dot{W^\alpha} (y) dy$ and $\dot{W^\alpha}(y)$ be a centered stationary Gaussian random field with covariance function $\kappa |x-y|^{-\alpha}$, where $\kappa = \kappa_\frakg V_1^2 > 0$ and $\kappa_\frakg$ and $V_1$ are defined as in \eqref{e.tail} and \eqref{e.odd}. Then we have the following.

\begin{theorem}
\label{thm:main3}
Suppose that the assumptions in Theorem \ref{t.size.l} hold. Let $G_u(x,y)$ be the Green's function of \eqref{e.Green}. Then as $\eps \to 0$,
\begin{equation}
\frac{u^\eps - u}{\sqrt{\eps^{\alpha}}} \xrightarrow[\eps \to 
0]{\mathrm{distribution}} \sqrt{\kappa} \int_{D} G_u(x, y) u (y) W^\alpha(dy), \quad\quad \text{in } L^2(D).
\label{e.main3}
\end{equation}
\end{theorem}

\begin{remark}\label{rem:main3}
We will not present the proof since they can easily be adapted from the approach in \cite{BGGJ12_AA,Jing15} and our control of the nonlinear terms earlier. The right hand side of \eqref{e.main3} is an integral with respect to the multiparameter Gaussian random processes $W^\alpha$, we refer to \cite{Khoshnevisan} for the theory. Let us denote this integral by $X$; then it determines a Gaussian distribution in the space $L^2(D)$. In particular, for any $\varphi \in L^2(D)$, the inner product $\langle \varphi, X\rangle$ has Gaussian distribution $\mathcal{N}(0,\sigma^2_{\alpha,\varphi})$ with
\begin{equation}
\sigma_{\alpha,\varphi}^2 := \kappa \int_{D^2} \frac{ (u \cG_u \varphi) 
(y) (u \cG_u \varphi)(z)}{|y-z|^\alpha} dy dz.
\label{e.covmat}
\end{equation}
\end{remark}

\medskip


\subsection{Non-separated nonlinearities}

We have so far assumed that the heterogeneous reaction function was of the form $f^\eps(x,u,\omega) = q_\eps(x,\omega)u + f(u)$; in other words, the random potential $q_\eps$ and the nonlinear function $f$ are separated. This choice is made mainly for notational simplicity. By some careful modifications of the main assumptions, we can extend our result to a class of general nonlinearities in non-separated form and consider a heterogeneous problem of the form:
\begin{equation} \label{e.grpde}
\left\{
\begin{aligned}
&-\Delta u^\eps + f\left(\frac{x}{\eps},u^\eps; \om\right) = g(x), \quad&\quad& x \in D,\\
&u^\eps = 0, \quad&\quad& x \in \partial D.
\end{aligned}
\right.
\end{equation}
Here, the nonlinear reaction is $f_\eps(x,u,\om) = f(\frac{x}{\eps},u,\om)$. We may next modify [A1][A2][A3] as follows.
\begin{itemize}
\item[{[A1']}] $f(x,s,\omega)$, $x \in \R^d$, $s \in \R$, is a random field on $(\Om,\cF,\bP)$ and stationary ergodic in $x$. In particular, there exists $\widetilde f(s,\om)$ such that $f(x,s,\om) = \widetilde{f}(s,\tau_x \om)$.

\item[{[A2']}] There exist $M \ge 1$ and $\gam > 1$ such that
\begin{equation}
\big| \widetilde{f}(s,\om) \big| \le M(1+|s|^\gam), \quad \text{ for all } s \in \R \text{ and } \om \in \Om.
\end{equation}
Moreover, if $d=3$, we further assume that $\gam < 2d/(d-2) -1 =  5$.

\item[{[A3']}] $\widetilde f(\cdot,\om): \R \to \R$ is continuously twice differentiable uniformly in $\om$, and for some $c > 0$, $\gam_1, \gam_2 \ge 0$, for all $\om \in \Om$ and $s \in \R$,
\begin{eqnarray}
&&\lambda_1 + \widetilde{f}'(s,\om) \ge c,\label{e.gfprime}\\
&&|\widetilde{f}'(s,\om)| \le M(1+|s|^{\gam_2}), \; |\widetilde{f}''(s,\om)| \le M(1+|s|^{\gam_2}). \label{e.gfpprime}
\end{eqnarray}
\end{itemize}
Here $\widetilde{f}'(s,\om)$ and $\widetilde{f}''(s,\om)$ denote, respectively, the first and second order derivatives of $\widetilde{f}$ with respect to $s$. Let $\ol{f}(s)$ be the mean $\E f(y,s,\cdot) = \E \widetilde{f}(s,\cdot)$. Then the mean-zero fluctuation of the random potential, i.e. $\nu(y,\om)$ defined in \eqref{e.nudef}, is replaced by
\begin{equation}
\label{e.gnudef}
\nu(y,s,\om) = f'(y,s,\om) - \ol{f}'(s).
\end{equation}
We also let $\mu(y,s,\om) := f(y,s,\om) - \ol{f}(s)$, which is the antiderivative of $\nu$ and $\mu_\eps(x,u)$ plays the role of $\nu_\eps(x)u$ in the separated case. 

To ensure that the random medium has short range correlations, we assume that:
\begin{itemize}
\item[{[S']}] There exists $\varrho: \R_+ \to \R$ satisfying [S] such that, for each $s \in \R$, the maximal correlation function of the family of random fields $\{\nu(x,s), \mu(x,s)\}$, $x \in K$, is bounded above by $\varrho$.
\end{itemize}

Under [A1'][A2'][A3'] and [S'], the approach presented in this paper may be applied to prove that $u^\eps$ of \eqref{e.grpde} converges weakly in $H^1$ and strongly in $L^2$, as $\eps \to 0$ for a.e. $\om \in \Om$, to the solution $u$ of
\begin{equation} \label{e.ghpde}
\left\{
\begin{aligned}
&-\Delta u + \ol{f}(u) = g(x), \quad&\quad& x \in D,\\
&u = 0, \quad&\quad& x \in \partial D.
\end{aligned}
\right.
\end{equation}
Indeed, the existence and well-posedness of the above equations are guaranteed by [A2'][A3']. We note that the proofs of Lemma \ref{l.well} and Lemma \ref{l.Linf} can be adapted since they only require the uniform lower bound on the derivative of $f$ and the uniform bound on growth rate of $f$. Hence, we still have uniform $L^\infty$ bound on $u^\eps$. The non-separated form does introduce some technicality in the proof of almost sure homogenization theory because, in the application of ergodic theorem, we have to deal with the $u$-dependent random process $f(y,u,\om)$. Even though $f(\frac{x}{\eps},s,\om)$ converges weakly in $L^p_{\rm loc}$ for all $p \in [1,\infty)$ to $\ol{f}(s)$ a.e. in $\Om$ for each fixed $s$, to show $f(\frac{x}{\eps},u(x))$ converges to $\ol{f}(u(x))$, one needs extra effort. For instance, a full measure event $\Om_1$ can be constructed by intersecting $\widetilde\Om_s$, over rationals $s \in \mathbb{Q}$, where $\widetilde\Om_s$ is the event on which $f(\frac{x}{\eps},s)$ converges to $\ol{f}(s)$. The proof of Lemma \ref{t.homog} then gets through by an additional approximation of $u$ using simple functions with rational values.

Once homogenization theory is established, we can continue to carry out the error estimates and the analysis of the distribution of fluctuations. The linearized equation, i.e. the analog to \eqref{e.Green}, is
\begin{equation}
\label{e.gGreen}
\left\{
\begin{aligned}
&-\Delta G_u(x;y) + \ol{f}'(u(x))\, G_u(x;y) = \delta_y, \quad &\quad &\text{ for } x \in D,\\
&G_u(x;y) = 0, \quad &\quad &\text{ for } x \in \partial D.
\end{aligned}
\right. 
\end{equation}
The leading term in the distribution will be given by the solution to
\begin{equation*}
-\Delta \chi^\eps + \ol{f}'(u) \chi^\eps = -\mu_\eps(x,u,\om), \quad \text{ in } D,
\end{equation*}
which can be conveniently written as $\chi^\eps = - \cG_u \mu_\eps(x,u)$. Here, $\cG_u$ denotes the fundamental solution operator to \eqref{e.gGreen}. We then get the expansion formula
\begin{equation}
\label{e.gexpan}
\begin{aligned}
u^\eps - u = & -\cG_u \left(\mu_\eps(x,u)\right)  - \cG_u \nu_\eps(x,u) (u^\eps - u) - \cG_u \left( f_\eps(x,u^\eps) - f_\eps(x,u) - f'_\eps(x,u)(u^\eps - u)\right)\\
= &  -\cG_u \left(\mu_\eps(x,u)\right) + \cG_u \nu_\eps(x,u) \cG_u \mu_\eps(x,u) +  \cG_u \nu_\eps(x,u) \cG_u \nu_\eps(x,u) (u^\eps - u) \\
 & - \cG_u \left( f_\eps(x,u^\eps) - f_\eps(x,u) - f'_\eps(x,u)\xi^\eps\right)\\
 & + \cG_u \nu_\eps(x,u) \cG_u \left( f_\eps(x,u^\eps) - f_\eps(x,u) - f'_\eps(x,u)\xi^\eps\right).
\end{aligned}
\end{equation}
In view of the modified assumption [S'], and by following the argument developed in earlier sections, we show that all items on the right except the first one are of negligible order in distribution. In particular, as an analog to Theorem \ref{t.sl.23}, we have
\begin{equation}
\label{e.gsl.23}
\frac{u^\eps - u}{\sqrt{\eps^d}} \xrightarrow{\rm{distribution}}  \int_D G_u (x,y) \sigma_\mu(u(y)) dW(y), \quad \text{ in  } L^2(D).
\end{equation}
Here, $\sigma_\mu(s)$, for each fixed $s \in \R$, is defined as
\begin{equation*}
\sigma^2_\mu(s) = \int_{\R^d} R_\mu(x,s) dx = \int_{\R^d} \E \mu(x,s)\mu(0,s) dx.
\end{equation*}


\subsection{Further studies}

We conclude this section by mentioning a couple of extensions to the above results. First, this paper considers only the physical dimension $d = 2,3$, and it would seem natural to extend the studies of this paper to arbitrary dimension $d \ge 4$. For linear equations, as studied by \cite{Jing15}, the framework of \cite{B-CLH-08,BJ-CMS-11,BGGJ12_AA} can still be applied, and is more or less unchanged, provided that we seek for the limiting distribution in $H^{-1}$ or other weaker spaces. This approach may not apply directly in the nonlinear setting. The uniform $L^\infty$ estimates on $u^\eps$ is essential in our treatment of the nonlinearity, and such an estimate is not available in higher dimension for general source term $g$, the right hand side in \ref{e.rpde}, in $L^2$ and not more regular. It remains to explore how to generalize the framework to other functional settings and to develop new ways to control the nonlinearity terms.

Another related further study is to modify the analysis in the continuum setting to handle the discrete setting, i.e., to address the numerical methods of \eqref{e.rpde}. As shown in \cite{BJ11_MMS,BJ14_M2AN} for the linear equations, studying the limiting distribution of the difference between the solutions to the heterogeneous equation and to the homogenized solution, obtained from multi-scale numerical schemes, in the limit of $\eps \to 0$ and then the discretization size $h \to 0$, and comparing the results with the theory in the continuum setting, one can build a benchmark to assess the performance of numerical schemes in capturing numerically the fluctuations of heterogeneous equations. To perform such an analysis in the nonlinear setting requires new ideas, in addition to those in \cite{BJ11_MMS,BJ14_M2AN}.


\bigskip

\section*{Acknowledgements}

GB acknowledges partial support from NSF grant DMS-1408867. WJ acknowledges partial support from NSF grant DMS-1515150. The authors thank the referees for their helpful comments.

\bigskip

\bibliographystyle{abbrv}
\bibliography{bj_rd}
\end{document}